\documentclass[a4paper,11pt,twoside]{amsart}
\usepackage[english]{babel}
\usepackage[utf8]{inputenc}

\usepackage[a4paper,inner=3cm,outer=3cm,top=3.8cm,bottom=3.8cm,pdftex]{geometry}
\usepackage{fancyhdr}
\usepackage{titlesec}
\usepackage{color}
\usepackage{bold-extra}
\titleformat{\section}{\normalfont\scshape\centering}{\thesection}{1em}{}
\titleformat{\subsection}{\bfseries}{\thesubsection}{1em}{}
  
\usepackage{comment}
\usepackage{graphics}
\usepackage{aliascnt}
\usepackage[pdftex,citecolor=green,linkcolor=red]{hyperref}

\usepackage{amsmath}
\usepackage{amsfonts}
\usepackage{amssymb}
\usepackage{amsthm}
\usepackage{comment}
\usepackage{mathtools}

\newtheorem{theorem}{Theorem}[section]

\newtheorem{lemma}[theorem]{Lemma}

\theoremstyle{definition}

\newtheorem{remark}[theorem]{Remark}

\numberwithin{equation}{section}

\newcommand{\Gal}{\textup{Gal}}

\address{Department of Mathematics and Statistics, P.O. Box 68, 00014 Helsinki, Finland}
\email{olli.jarviniemi@helsinki.fi}

\title{On the Common Prime Divisors of Polynomials}

\date{}
\author{Olli Järviniemi}

\begin{document}
\begin{abstract}
The prime divisors of a polynomial $P$ with integer coefficients are those primes $p$ for which $P(x) \equiv 0 \pmod{p}$ is solvable. Our main result is that the common prime divisors of any several polynomials are exactly the prime divisors of some single polynomial. By combining this result with a theorem of Ax we get that for any system $F$ of multivariate polynomial equations with integer coefficients, the set of primes $p$ for which $F$ is solvable modulo $p$ is the set of prime divisors of some univariate polynomial. In addition, we prove results on the densities of the prime divisors of polynomials. The article serves as a light introduction to algebraic number theory and Galois theory.\end{abstract}

\maketitle

\section{Introduction}
\label{sec:intro}

Let $P$ be a polynomial with integer coefficients. A prime $p$ is said to be a prime divisor of $P$ if $p$ divides $P(n)$ for some integer $n$. What can we say about the prime divisors of $P$?

By an elementary argument one can prove that the set of prime divisors of $P$, denoted by $S(P)$, is infinite for non-constant $P$ (see \cite{MT}, Theorem 2 for a proof). Much more is true. The celebrated Chebotarev density theorem (or the weaker Frobenius density theorem) proves that the density of $S(P)$ with respect to the set of prime numbers is in fact positive (see \cite{SL}). This is discussed in more detail later.

For linear polynomials $P(x) = ax + b$ one easily sees that $p \in S(P)$ for all $p$ except for those which divide $a$ and not $b$. For polynomials $P$ of degree two there also exists an explicit description of the set $S(P)$, which can be proven by applying theory of quadratic residues. For example, the set $S(x^2 + 1)$ consists of all primes $p$ congruent to $1 \pmod{4}$, and in addition the prime $p = 2$. In the general case of an arbitrary $P$ no explicit description of the set $S(P)$ is known. The characterization of $S(P)$ by congruence conditions has been considered in \cite{Suresh}. See \cite{Dalawat} for discussion on the general case.

One may then ask about the characterization of other similarly defined sets. For example, what does $S(A) \cap S(B)$ look like for polynomials $A$ and $B$? It has been proved by Nagell that $S(A) \cap S(B)$ is infinite for any non-constant $A$ and $B$ (\cite{MT}, Theorem 3), and the proof generalizes for an arbitrary number of polynomials.\footnote{See \cite{Nagell} for the original proof by Nagell.} In (\cite{GB}, Theorem 10) it has been proved that for many types of $A$ and $B$ there exists such a $D \in \mathbb{Z}[x]$ that $S(A) \cap S(B) = S(D)$. Our main result is that in fact for any $A$ and $B$ there exists such a $D$.

Note that the set $S(A) \cap S(B)$ is the set of those primes $p$ for which the system
\[
\begin{cases}
A(x) \equiv 0 \pmod{p} \\
B(y) \equiv 0 \pmod{p}
\end{cases}
\]
of congruences has an integer solution. In general, for multivariate polynomials $F_1, F_2, \ldots , F_m$ in the variables $x_1, \ldots , x_n$ one can consider the set of primes $p$ for which the system
\[
\begin{cases}
F_1(x_1, \ldots , x_n) \equiv 0 \pmod{p} \\
F_2(x_1, \ldots , x_n) \equiv 0 \pmod{p} \\
\vdots \\
F_m(x_1, \ldots , x_n) \equiv 0 \pmod{p} \\
\end{cases}
\]
has an integer solution. By a combination of our main result above and a result of Ax (\cite{Ax}, Theorem 1) we prove that the set of these $p$ is again merely the set $S(D)$ for some polynomial $D$ (in one variable) with integer coefficients.

The structure of the article is as follows. We first define a bit of notation below. We then state our results. We give an intuitive explanation on why the main theorem is true after presenting the necessary preliminaries. We then prove the main theorem. Finally, some density problems are considered, starting by introducing the Frobenius density theorem.

By $\mathbb{Z}[x]$ we denote the set of polynomials with integer coefficients, and $\mathbb{Z}_c[x]$ denotes the non-constant polynomials of $\mathbb{Z}[x]$. In general, $K[x]$ denotes the set of polynomials with coefficients from $K$. The letter $p$ will always denote a prime number, and the letters $n, m, k$ refer to integers.

The (natural) density $\delta(S)$ of a set $S$ consisting of primes is defined as the limit
$$\lim_{x \to \infty} \frac{|\{p : p \le x, p \in S\}|}{|\{p : p \le x, p \text { is prime}\}|}.$$
This limit does not necessarily exist. However, for the sets we consider the limit always exists, as we will see from the Frobenius density theorem.

For further reading on related topics we suggest \cite{Ballot}, \cite{Hooley}, \cite{Lagarias} and \cite{Morton}.

\section{Results}
\label{sec:results}

As mentioned above, the main result concerns the common prime divisors of several polynomials.

\begin{theorem}
\label{thm:intersection}
Let $n$ be a positive integer, and let $P_1, P_2, \ldots , P_n \in \mathbb{Z}[x]$ be arbitrary. There exists a polynomial $D \in \mathbb{Z}[x]$ such that
$$S(D) = S(P_1) \cap S(P_2) \cap \ldots \cap S(P_n).$$
\end{theorem}
The proof of the theorem is constructive. Furthermore, the constructed $D$ has degree $T = \deg(P_1) \cdots \deg(P_n)$, which in some cases is the minimal possible. 


The next theorem is the generalization for systems of polynomial congruences.

\begin{theorem}
\label{thm:system}
Let $n$ and $m$ be positive integers. Let $F_1, F_2, \ldots , F_m$ be polynomials in the variables $x_1, \ldots , x_n$ with integer coefficients. There exists a polynomial $D \in \mathbb{Z}[x]$ such that the system
\[
\begin{cases}
F_1(x_1, \ldots , x_n) \equiv 0 \pmod{p} \\
F_2(x_1, \ldots , x_n) \equiv 0 \pmod{p} \\
\vdots \\
F_m(x_1, \ldots , x_n) \equiv 0 \pmod{p} \\
\end{cases}
\]
of congruences has a solution if and only if $p \in S(D)$.
\end{theorem}

The next results concern the densities $\delta(S(P))$. The first one proves a lower bound for the density of the common prime divisors of several polynomials.
\begin{theorem}
\label{thm:bound}
Let $P_1, \ldots , P_n \in \mathbb{Z}_c[x]$ be given. We have
$$\delta\left(S(P_1) \cap \ldots \cap S(P_n)\right) \ge \frac{1}{\deg(P_1) \cdots \deg(P_n)}.$$
\end{theorem}
We also give a family of non-trivial examples of equality cases.

We then prove that one can under additional assumptions calculate the density of $S(P_1) \cap \ldots \cap S(P_n)$. These assumptions concern the splitting fields of $P_i$ and their degrees, which are defined later on.

\begin{theorem}
\label{thm:product}
Let $P_1, \ldots , P_n \in \mathbb{Z}_c[x]$ be given. Let $F_i$ be the splitting field of $P_i$, and let $F$ be the compositum of the fields $F_i$. Assume that we have
$$[F : \mathbb{Q}] = [F_1 : \mathbb{Q}] \cdots [F_n : \mathbb{Q}].$$
Then,
$$\delta\left(S(P_1) \cap \ldots \cap S(P_n) \right) = \delta(S(P_1)) \cdots \delta(S(P_n)).$$
\end{theorem}
This density is what one would expect, and corresponds to the definition of the independence of random variables.

Finally, we prove that all rational numbers between $0$ and $1$ occur as densities. The converse direction has already been solved by the Frobenius density theorem.

\begin{theorem}
\label{thm:surjective}
Let $r$ be a rational number between $0$ and $1$. There exists a polynomial $P \in \mathbb{Z}[x]$ such that
$$\delta(S(P)) = r.$$
\end{theorem}

\section{Preliminaries}

We refer the reader to the textbooks of Lang \cite{Lang} and van der Waerden \cite{Waerden} for more comprehensive expositions of the theory of field extensions. Here we only present the prerequisites needed to prove our results.

A complex number $\alpha$ is called an algebraic number if there exists some non-constant $P \in \mathbb{Q}[x]$ such that $P(\alpha) = 0$. To an algebraic number $\alpha$ we associate its minimal polynomial. This is the non-constant monic polynomial $P \in \mathbb{Q}[x]$ with the minimum possible degree such that $P(\alpha) = 0$. For example, the minimal polynomial of $\sqrt{2}$ is $x^2 - 2$. Any rational number $q$ is algebraic with minimal polynomial $x - q$.

Note that the minimal polynomial $P$ is unique. Indeed, if $P_1$ and $P_2$ are non-constant monic polynomials in $\mathbb{Q}[x]$ such that $P_1(\alpha) = P_2(\alpha) = 0$, $\deg(P_1) = \deg(P_2)$ and $P_1 \neq P_2$, then $P_1 - P_2$ is non-constant, has degree smaller than $\deg(P_1)$, satisfies $P_1(\alpha) - P_2(\alpha) = 0$ and is monic after multiplying by a suitable constant.

If $P$ is the minimal polynomial of $\alpha$, then for any polynomial $Q$ with rational coefficients satisfying $Q(\alpha) = 0$ we have $P | Q$. This is proven by the division algorithm: write $Q(x) = P(x)T(x) + R(x)$ with $\deg(R) < \deg(P)$. We have $R(\alpha) = 0$, so by the minimality of the degree of $P$ we must have $R = 0$.

We then prove that the minimal polynomial $P$ of $\alpha$ is irreducible over $\mathbb{Q}$. Assume the contrary, and write $P(x) = Q(x)R(x)$ for $Q, R \in \mathbb{Q}[x]$ non-constant. We may assume that $Q$ and $R$ are monic. Now $0 = P(\alpha) = Q(\alpha)R(\alpha)$, so either $Q(\alpha) = 0$ or $R(\alpha) = 0$. This contradicts the minimality of the degree of $P$. Conversely, if $P \in \mathbb{Q}[x]$ is monic, irreducible and satisfies $P(\alpha) = 0$, then $P$ is the minimal polynomial of $\alpha$, following from the result of the previous paragraph.

For an algebraic number $\alpha$ we denote by $\mathbb{Q}(\alpha)$ the set of numbers of the form
$$q_0 + q_1\alpha + \ldots + q_{d-1}\alpha^{d-1},$$
where $q_0, \ldots , q_{d-1}$ are rational numbers and $d$ is the degree of the minimal polynomial $P$ of $\alpha$. A sum of two such numbers is of such form, and it is not hard to see that the product of such numbers is also of such form, as from the equation $P(\alpha) = 0$ we obtain a method to calculate $\alpha^n$ in terms of $1, \alpha, \ldots , \alpha^{d-1}$ for all $n \ge d$. More formally, one may prove by a straight-forward induction that $\alpha^n$ may be expressed as a linear combination of $1, \alpha, \ldots , \alpha^{d-1}$ for any $n \ge d$.

Also, if $x \in \mathbb{Q}(\alpha)$ is not equal to zero, we have $\frac{1}{x} \in \mathbb{Q}(\alpha)$. We present a direct proof. We first claim that there exists rationals $q_0, q_1, \ldots , q_d$, not all zero, such that
$$q_0 + q_1x + \ldots + q_dx^d = 0.$$
Indeed, each power of $x$ may be written as a linear combination of the $d$ numbers $1, \alpha, \ldots , \alpha^{d-1}$. This linear combination can be thought as a $d$-dimensional vector. Since no $d+1$ vectors of dimension $d$ can be linearly independent, there are rationals $q_i$ satisfying the conditions. Let $i$ be the smallest index for which $q_i \neq 0$. We now have
$$q_{i+1}x^{i+1} + \ldots + q_dx^{d} = -q_ix^i.$$
Since $x \neq 0$ and $q_i \neq 0$, we may divide both sides by $-q_ix^{i+1}$. The left hand side is now a polynomial in $x$, which we know to be in $\mathbb{Q}(\alpha)$, while the right hand side is $\frac{1}{x}$. Thus, $\frac{1}{x} \in \mathbb{Q}(\alpha)$.

We may write the above results as follows: $\mathbb{Q}(\alpha)$ is a field. Recall that a field is a set of elements equipped with addition and multiplication. The addition and multiplication are assumed to be associative and commutative, to have neutral elements and to have inverse elements. Finally, we connect additivity and multiplicativity by the distributive law: $a(b+c) = ab + ac$.

We further emphasise the fact that $\mathbb{Q}(\alpha)$ contains the field of rational numbers $\mathbb{Q}$ inside of it by saying that $\mathbb{Q}(\alpha)$ is a field extension of $\mathbb{Q}$. To a field extension $\mathbb{Q}(\alpha)$ we define its degree as the degree of the minimal polynomial of $\alpha$. The degree of $\mathbb{Q}(\alpha)$ is denoted by $[\mathbb{Q}(\alpha) : \mathbb{Q}]$. Intuitively, the degree of an extension tells how ``large'' the extension is.

So far we have only focused on the rational numbers and its extensions. We may, however, further extend the field $\mathbb{Q}(\alpha)$ by adjoining an algebraic number $\beta$ to it. One may define a minimal polynomial over $\mathbb{Q}(\alpha)$ similarly as we did above for $\mathbb{Q}$ and prove that the set of polynomial expressions of $\alpha$ and $\beta$ form a field. We omit the details, as the proofs are so similar to the ones above. The obtained field is denoted by $\mathbb{Q}(\alpha, \beta)$.

The primitive element theorem tells that for any algebraic numbers $\alpha$ and $\beta$ the field $\mathbb{Q}(\alpha, \beta)$ has the same elements as $\mathbb{Q}(\gamma)$ for some suitable choice of $\gamma$. Thus, adjoining two algebraic numbers to $\mathbb{Q}$ may be done with just one algebraic number, and by induction this result generalizes for an arbitrary number of algebraic numbers.

We conclude this section by proving the primitive element theorem. The proof we give is based on the one given in (\cite{Waerden}, Section 6.10).  Let $\alpha$ and $\beta$ be given algebraic numbers. Let $\lambda$ be a rational parameter, and define $\gamma = \alpha + \lambda\beta$. We prove that for all except finitely many choices of $\lambda$ we have $\mathbb{Q}(\gamma) = \mathbb{Q}(\alpha, \beta)$, which proves the theorem. Note that $\mathbb{Q}(\gamma)$ is a subfield of $\mathbb{Q}(\alpha, \beta)$, so we focus on proving the other inclusion. For this it suffices to prove $\beta \in \mathbb{Q}(\gamma)$, as we then also have $\alpha = \gamma - \lambda\beta \in \mathbb{Q}(\gamma)$.

Let $F$ be the minimal polynomial of $\beta$ over $\mathbb{Q}(\gamma)$. We will prove that $F$ has degree $1$, which implies the claim. The idea is to construct two polynomials both of which $F$ divides and whose greatest common divisor has degree $1$.

Let $f$ and $g$ be the minimal polynomials of $\alpha$ and $\beta$ over $\mathbb{Q}$. Since $g \in \mathbb{Q}[x] \subset \mathbb{Q}(\gamma)[x]$ and $g(\beta) = 0$, $g$ is divisible by $F$. If we now define
$$h(x) = f(\gamma - \lambda x) \in \mathbb{Q}(\gamma)[x],$$
then $h(\beta) = f(\alpha) = 0$, so $h$ too is divisible by $F$.

It remains to show that the greatest common divisor of $g$ and $h$ has degree one. Assume not. Since $g$ is the minimal polynomial of $\beta$ over $\mathbb{Q}$, it can't have $\beta$ as its double root (as otherwise the derivative $g'$ of $g$ would have $\beta$ as its root, and thus $g$ wouldn't be the minimal polynomial of $\beta$). Thus, there must be some $\beta' \neq \beta$ for which $g(\beta') = h(\beta') = 0$. By the definition of $h$ this implies
$$f(\alpha + \lambda\beta - \lambda\beta') = 0,$$
so $\alpha + \lambda\beta - \lambda\beta'$ equals $\alpha'$ for some root $\alpha'$ of $f$. We now get
$$\lambda = \frac{\alpha' - \alpha}{\beta - \beta'}.$$
Since the roots $\alpha'$ and $\beta'$ attain only finitely many values, we have obtained a contradiction assuming $\lambda$ does not belong to a finite set of exceptional values.

\section{Motivation behind Theorem \ref{thm:intersection}}
\label{sec:motivation}

Let $A(x) = x^2 - 2$ and $B(x) = x^2 + 1$. We prove that there exists a polynomial $D \in \mathbb{Z}[x]$ such that $S(A) \cap S(B) = S(D)$.

Note that $\alpha := \sqrt{2}$ is a root of $A$ and $\beta := i$ is a root of $B$. We choose $D$ to be the minimal polynomial of $\sqrt{2} + i$. We first prove that $D(x) = x^4 - 2x^2 + 9$. Note that $D$ is monic and irreducible.\footnote{There are many ways to prove irreducibility. We give a direct argument. By the rational root theorem one sees that $D$ has no rational roots. Thus, if $D$ is reducible over $\mathbb{Q}$, it must have two second degree factors $P$ and $Q$. Since $D(\sqrt{2} + i) = 0$, one of the factors, say $P$, must have $\sqrt{2} + i$ as its root. On the other hand, by the quadratic formula we get that the roots of $P$ are of the form $a \pm \sqrt{b}$, where $a$ and $b$ are some rational numbers. By comparing imaginary parts we get that $b = -1$. This implies that $a = \sqrt{2}$, which contradicts the irrationality of $\sqrt{2}$.} We now check that $D$ has $\alpha + \beta$ as its root.

For simplicity, let $\gamma = (\alpha + \beta)^2$. Opening the brackets and using $\alpha^2 = 2$ and $\beta^2 = -1$ we get
$$\gamma = 2\alpha\beta + 1.$$
Now,
$$D(\alpha + \beta) = \gamma^2 - 2\gamma + 9 = (-7 + 4\alpha\beta) - 2(2\alpha\beta + 1) + 9 = 0,$$
as desired.

We now prove one direction of our claim, namely that $S(A) \cap S(B) \subset S(D)$. Let $p \in S(A) \cap S(B)$ be given, and let $A(a) \equiv B(b) \equiv 0 \pmod{p}$. The idea is that $a$ corresponds to ``$\sqrt{2}$ modulo $p$'', and similarly $b$ corresponds to $i$. Indeed, we may replicate the proof above in this context. Let $c = (a+b)^2$. By using $a^2 \equiv 2 \pmod{p}$ and $b^2 \equiv -1 \pmod{p}$ we get
$$c \equiv 2ab + 1 \pmod{p}.$$
Therefore,
$$D(a+b) \equiv c^2 - 2c + 9 \equiv (-7 + 4ab) - 2(2ab + 1) + 9 \equiv 0 \pmod{p},$$
so $p \in S(D)$.

The other direction of the claim is a bit trickier. The proof above relies on the fact that if $a$ and $b$ are roots of $A$ and $B$ modulo $p$, then we may generate from $a$ and $b$ a root of $D$ modulo $p$ (simply by picking $a+b$). For the other direction we pick some root $d$ of $D$ modulo $p$ and generate a root of $A$ modulo $p$, and similarly for $B$. The method is not our creation, but may be found in (\cite{GB}, Theorem 2).

The idea is that the root $\alpha + \beta$ of $D$ was chosen so that $\mathbb{Q}(\alpha, \beta) = \mathbb{Q}(\alpha + \beta)$ (recall the primitive element theorem). Therefore, there exists some polynomial $A_*$ such that
$$A_*(\alpha + \beta) = \alpha.$$
In this case we may explicitly calculate that
$$A_*(x) = \frac{5x - x^3}{6}.$$
Now, the polynomial $A(A_*(x))$ has $\alpha + \beta$ as its root, so it is divisible by the minimal polynomial $D$ of $\alpha + \beta$. Therefore,
$$A(A_*(x)) = D(x)E(x)$$
for some polynomial $E \in \mathbb{Q}[x]$, which in this case equals $\frac{1}{36}(x^2 - 8)$.

Now, if $p \in S(D)$ and $d$ is an integer such that $D(d) \equiv 0 \pmod{p}$, then we have
$$A(A_*(d)) = D(d)E(d) = D(d) \cdot \frac{1}{36}(d^2 - 8).$$
For $p = 2$ and $p = 3$ the denominator $36$ is divisible by $p$, which causes problems. For other $p$ we get that the equation $D(d) \equiv 0 \pmod{p}$ implies $D(d) \cdot \frac{1}{36}(d^2 - 8) \equiv 0 \pmod{p}$. This proves that $A(A_*(d)) \equiv 0 \pmod{p}$, so $A$ has a root $A_*(d) = \frac{5d-d^3}{6}$ modulo $p$. Again, the denominator causes problems for $p = 2$ and $p = 3$, but for other values this argument works.

\hrulefill

It is worth to take a moment to convince the reader that one may handle divisibility of rational numbers in this way. For a rational number $\frac{a}{b}$ with $\gcd(a, b) = 1$ and a positive integer $m$ we define $m \mid \frac{a}{b}$ if $m \mid a$. (Note that this implies $\gcd(m, b) = 1$.) For convenience, we always assume that the rational numbers are written in their lowest terms (with $n = \frac{n}{1}$ for an integer $n$).

The following lemmas follow easily from the definition.

\begin{lemma}
\label{lem:divSum}
Let $m$ be a positive integer and let $r_1$ and $r_2$ be rational numbers such that $m \mid r_1$ and $m \mid r_2$. Then $m \mid r_1 + r_2$.
\end{lemma}

\begin{lemma}
\label{lem:divProd}
Let $m$ be a positive integer and let $r_1$ and $r_2$ be rational numbers. Assume that $m \mid r_1$ and that the greatest common divisor of $m$ and the denominator of $r_2$ equals $1$. Then $m \mid r_1r_2$.
\end{lemma}

The most difficult result we use related to this generalized definition is the following.

\begin{lemma}
\label{lem:ratRoot}
Let $P \in \mathbb{Z}[x]$ and let $p$ be a prime number. Assume that $p \mid P(\frac{a}{b})$ and $p \nmid b$. Then $p \in S(P)$.
\end{lemma}

In other words, a rational root modulo $p$ implies an integer root modulo $p$, assuming the numerator is not divisible by $p$. 

\begin{proof}
Let $n = ab^{p-2}$.\footnote{Note that, by Fermat's little theorem, $b^{p-2}$ is the inverse element of $b$ modulo $p$. Thus, one may think of $n$ as ``$\frac{a}{b}$ modulo $p$''.} We prove that $P(n) \equiv 0 \pmod{p}$. Write $P(x) = c_dx^d + \ldots + c_1x + c_0$. Since $p \mid P(\frac{a}{b})$, the number
$$b^dP\left(\frac{a}{b}\right) = \sum_{i = 0}^d c_ia^ib^{d-i}$$
is divisible by $p$. The right hand side is an integer. Consider it modulo $p$ and use Fermat's little theorem to get
$$0 \equiv \sum_{i = 0}^d c_ia^ib^{d-i} \equiv b^d \cdot \sum_{i = 0}^d c_ia^ib^{(p-2)i} \pmod{p}.$$
Since $p \nmid b$, we have
$$0 \equiv \sum_{i = 0}^d c_ia^ib^{(p-2)i} = P(n) \pmod{p},$$
and therefore $p \in S(P)$.
\end{proof}

\hrulefill

We have therefore proven that for all $p \neq 2, 3$ we have that $p \in S(D)$ implies $p \in S(A)$ (and we similarly prove $p \in S(B)$). Looking back, we do have $3 \in S(D)$, but $3 \not\in S(A) \cap S(B)$. This, however, can be dealt with by an elementary argument. Define
$$D_3(x) = \frac{D(9x)}{9},$$
so $D_3 \in \mathbb{Z}[x]$, and $3 \not\in S(D_3)$, as all coefficients of $D_3$ are divisible by $3$ except for the constant term. One also sees that $p \in S(D_3)$ if and only if $p \in S(D)$ for $p \neq 3$. Thus, one may ``fix'' the polynomial $D$ to not have $3$ as its prime divisor.

The proof of the general case proceeds along similar lines, and many steps do not differ much from the special case above. However, there are two steps which do not immediately generalize. First, we must formalize the concept of considering ``$\alpha$ modulo $p$'' for an algebraic number $\alpha$. Secondly, we must formalize the proof of the part $S(A) \cap S(B) \subset S(D)$. There exists cases that are more difficult than the one considered above, such as
$$A(x) = (x^2 - 123)^2 - 789$$
and
$$B(x) = (x^2 - 456)^2 - 789,$$
where we have
$$\alpha^2 = \beta^2 - 333$$
(for a suitable choice of $\alpha$ and $\beta$). This is a non-trivial relation between $\alpha$ and $\beta$ that has to be taken into account when simplifying polynomial expressions of $\alpha$ and $\beta$.

\section{Proof of Theorem \ref{thm:intersection}}

We may assume $n = 2$, as the general case follows by induction. Define $A = P_1$ and $B = P_2$. We first consider the case when $A$ and $B$ are irreducible over $\mathbb{Q}$, and prove the existence of an ``almost working'' $D \in \mathbb{Z}[x]$, meaning that $S(D)$ and $S(A) \cap S(B)$ differ by only finitely many primes.

\subsection{Almost working $D$ for irreducible $A$ and $B$}

Our first results concern formalizing the idea of considering ``$\alpha$ modulo $p$''.

\begin{lemma}
\label{lem:homoHelp}
Let $P \in \mathbb{Z}[x]$ be irreducible (over $\mathbb{Q}$), $P(\alpha) = 0$ and $P(a) \equiv 0 \pmod{p}$ for some prime $p$ and $a \in \mathbb{Z}$. Assume that $p$ does not divide the leading coefficient of $P$. Then, for all $Q \in \mathbb{Z_c}[x]$ with $Q(\alpha) = 0$ we have $Q(a) \equiv 0 \pmod{p}$.
\end{lemma}

\begin{proof}

Let $0 \neq b \in \mathbb{Z}$ be the leading coefficient of $P$. Since $P$ is irreducible, the minimal polynomial of $\alpha$ is $\frac{P}{b}$. The polynomial $Q$ has $\alpha$ as its root, so $Q$ is divisible by the minimal polynomial of $\alpha$, and therefore $\frac{P}{b} \mid Q$. Therefore there exists $R \in \mathbb{Q}[x]$ for which $Q = PR$. We shall prove by a proof by contradiction that the denominators of the coefficients of $R$ are not divisible by $p$. As $Q(a) = P(a)R(a)$, we are then done by Lemma \ref{lem:divProd}.

Let $P(x) = p_{d_P}x^{d_P} + \ldots + p_1x + p_0$ and $R(x) = r_{d_R}x^{d_R} + \ldots + r_0$. Assume that there exists such an index $i$ that $p$ divides the denominator of $r_i$, and pick the greatest $i$ satisfying this condition. We look at the coefficient of the term $x^{d_P + i}$ in $Q$. This is an integer which is by $Q = PR$ equal to $r_ip_{d_P} + r_{i+1}p_{d_P - 1} + \ldots r_{i + d_P}p_0$, where we define $r_j = 0$ for $j > d_R$. The denominator of the rational number $r_{i+1}p_{d_P - 1} + \ldots + r_{i+d_P}p_0$ is not divisible by $p$, but the denominator of $r_ip_{d_P}$ is, so the sum is not an integer. This is a contradiction, and therefore the denominator of $r_i$ is not divisible by $p$ for any $i$.
\end{proof}

Let $\alpha$ be some algebraic number and let $T = \{1, \alpha, \alpha^2, \ldots , \alpha^{\deg(P) - 1}\}$ be a basis of the $\mathbb{Q}$-vector space $\mathbb{Q}(\alpha)$, where $P$ is the minimal polynomial of $\alpha$. Define $\mathbb{Z}^p(\alpha)$ to be the set of those $r \in \mathbb{Q}(\alpha)$ whose representation by the basis $T$ does not contain any rational number whose denominator is divisible by $p$. If $p$ does not divide the denominator of any coefficient of $P$, the numbers $\mathbb{Z}^p(\alpha)$ form a ring, as one may express $\alpha^n$ in terms of the elements of $T$ for $n \ge \deg(P)$ by repeated simplifying by using the equation $P(\alpha) = 0$. A rigorous proof can be given by proving $\alpha^n \in \mathbb{Z}^p(\alpha)$ by induction for all $n$, which we leave to the reader.

The next lemma proves the existence of a homomorphism from $\mathbb{Z}^p(\alpha)$ to the set $\mathbb{F}_p$ of integers modulo $p$. This result is critical, as it formalizes the idea of considering algebraic numbers modulo primes mentioned earlier.

\begin{lemma}
\label{lem:homo}
Let $P \in \mathbb{Z_c}[x]$ be irreducible (over $\mathbb{Q}$). Let $P(a) \equiv 0 \pmod{p}$ for some prime $p$ and $a \in \mathbb{Z}$ , where $p$ does not divide the leading coefficient of $P$. Let $P(\alpha) = 0$. There exists a ring homomorphism $\phi : \mathbb{Z}^p(\alpha) \to \mathbb{F}_p$ for which $\phi(\alpha) = a$.
\end{lemma}

\begin{proof}

Every $x \in \mathbb{Z}^p(\alpha)$ can be represented as $x = Q_x(\alpha)$ for some $Q_x \in \mathbb{Q}[x]$, where the denominators of the coefficients of $Q_x$ are not divisible by $p$. We define $\phi(Q_x(\alpha)) \equiv Q_x(a) \pmod{p}$. Obviously this map is a homomorphism for which $\phi(\alpha) = a$. However, since the polynomial $Q_x$ is not unique, we have to check that $\phi$ is well defined.

Assume $x = Q_1(\alpha) = Q_2(\alpha)$. We want to prove $Q_1(a) \equiv Q_2(a) \pmod{p}$. Let $Q := Q_1 - Q_2$, so we want $Q(a) \equiv 0 \pmod{p}$ for all $Q(\alpha) = 0$. We may multiply $Q$ with the product of the denominators of its coefficients, which is not divisible by $p$. The problem is now reduced to the case $Q \in \mathbb{Z}[x]$, which has been handled in Lemma \ref{lem:homoHelp}.
\end{proof}

We now turn to proving that there exists a polynomial $D \in \mathbb{Z}[x]$ such that $S(A) \cap S(B)$ and $S(D)$ differ by only finitely many primes, when $A$ and $B$ are irreducible polynomials with integer coefficients. In what follows, $\alpha$ is some (fixed) root of $A$. Write $B = E_1E_2 \ldots E_t$, where $E_i \in \mathbb{Q}(\alpha)[x]$ are non-constant polynomials irreducible in $\mathbb{Q}(\alpha)$, and let $\beta_i$ be some root of $E_i$. By the primitive element theorem (or rather the proof of it), for all $1 \le i \le t$ there exists $n_i \in \mathbb{Z}$ for which $\mathbb{Q}(\alpha, \beta_i) = \mathbb{Q}(\alpha + n_i\beta_i)$. Let $D_i$ be the minimal polynomial of $\alpha + n_i\beta_i$ over $\mathbb{Q}$ multiplied by a suitable non-zero integer so that $D_i \in \mathbb{Z}[x]$.

For our desired polynomial $D$ we choose $D = D_1D_2 \ldots D_t$, and we now prove that this choice works. This is done by the following two lemmas.

\begin{lemma}
\label{lem:ABsubD}
Let $A, B$ and $D$ be as above. For all but finitely many $p$ for which $p \in S(A) \cap S(B)$ we have $p \in S(D)$.
\end{lemma}

\begin{proof}

Let $p \in S(A) \cap S(B)$ be given. We discard those finitely many primes $p$ which divide the leading coefficient of $A$. Let $a$ and $b$ be integers such that $A(a) \equiv B(b) \equiv 0 \pmod{p}$, and let $\phi$ be a ring homomorphism $\mathbb{Z}^p(\alpha) \to \mathbb{F}_p$ for which $\phi(\alpha) = a$, whose existence is guaranteed by Lemma \ref{lem:homo} (as $p$ does not divide the leading coefficient of $A$).

Consider the polynomial $D_i(\alpha + n_ix)$. This is a polynomial whose coefficients are in the field $\mathbb{Q}(\alpha)$ and whose one root is $\beta_i$. Therefore,\footnote{We proved this implication for only $\mathbb{Q}$ before, but the proof immediately generalizes to minimal polynomials over any field.} it is divisible by the minimal polynomial of $\beta_i$ (in $\mathbb{Q}(\alpha)$), that is, by $E_i$. Thus, there exists $F_i \in \mathbb{Q}(\alpha)[x]$ for which $D_i(\alpha + n_ix) = E_i(x)F_i(x)$. Let $D_*(x) := \prod_{i = 1}^t D_i(\alpha + n_ix)$, and let $F \in \mathbb{Q}(\alpha)[x]$ be the product of $F_i$. We now have $D_*(x) = B(x)F(x)$. 

Discard those finitely many primes $p$ for which the coefficients of $F(x)$ do not all belong to $\mathbb{Z}^p(\alpha)$. (Note that $F$ does not depend on $p$.) We now have
\begin{align*}
0 \equiv B(b) \equiv \phi(B(b)) \equiv \phi(B(b))\phi(F(b)) \equiv \phi(B(b)F(b)) \equiv \phi(D_*(b)) \\
\equiv  \prod_{i = 1}^t \phi(D_i(\alpha + n_ib)) \equiv  \prod_{i = 1}^t D_i(a + n_ib) \pmod{p}.
\end{align*}
Thus, we have $p \in S(D_i)$ for some $D_i$, and therefore $p \in S(D)$. As we have discarded only finitely many primes along the way, this proves the lemma.
\end{proof}

The next lemma finishes the proof for irreducible $A$ and $B$. As stated before, the idea of the proof is not ours but may be found in \cite{GB}.

\begin{lemma}
\label{lem:DsubAB}
Let $A, B$ and $D$ be as above. For all except finitely many $p$ for which $p \in S(D)$ we have $p \in S(A) \cap S(B)$.
\end{lemma}

\begin{proof}
Assume $p \in S(D)$, so $p \in S(D_i)$ for some $i$. Since $\mathbb{Q}(\alpha) \subset \mathbb{Q}(\gamma_i)$, there exists a polynomial $A_* \in \mathbb{Q}[x]$ such that $A_*(\gamma_i) = \alpha$. Since $A(A_*(x))$ has $\gamma_i$ as its root, for some polynomial $F \in \mathbb{Q}[x]$ we have $A(A_*(x)) = D_i(x)F(x)$, as $D_i$ is a scalar multiple of the minimal polynomial of $\gamma_i$ . It follows that for all except finitely many $p$ we have $p \in S(A)$. Similarly, for all except finitely many $p$ we have $p \in S(B)$ if $p \in S(D_i)$. This proves the lemma.
\end{proof}

\subsection{Extending to the general case}

We first prove that one may modify the prime divisors of a polynomial by any finite set of primes.

\begin{lemma}
\label{lem:fix}
Let $P \in \mathbb{Z}[x]$ and $p$ be given. There exists polynomials $P_+$ and $P_-$ in $\mathbb{Z}[x]$ with the following properties: a prime $q \neq p$ is in either all of $S(P), S(P_+)$ and $S(P_-)$ or in none of them, $p \in S(P_+)$ and $p \not\in S(P_-)$.
\end{lemma}

\begin{proof}
For $P_+$ we may simply take $P_+(x) = pP(x)$. For $P_-$ we first note that if $P(0) = 0$, then $S(P)$ contains all primes and we may take $P_-(x) = px + 1$. Otherwise, let $p^k$ be the largest power of $p$ dividing $P(0)$. Now, define
$$P_-(x) = \frac{P(p^{k+1}x)}{p^k}.$$
A direct calculation gives that $P_-$ has integer coefficients, and all coefficients except for the constant term of $P_-$ are divisible by $p$. Therefore, $p \not\in S(P_-)$. It is also easy to check that $q \in S(P_-)$ if and only if $q \in S(P)$, when $q \neq p$ is a prime, from which the claim follows.
\end{proof}

We have thus proven Theorem \ref{thm:intersection} for irreducible $A$ and $B$. We now extend this result to the general case. We first note that we may take unions of sets of the form $S(P)$.

\begin{lemma}
\label{lem:union}
Let $P_1, \ldots , P_n$ be polynomials with integer coefficients. There exists a polynomial $D \in \mathbb{Z}[x]$ such that
$$S(D) = S(P_1) \cup \ldots \cup S(P_n).$$
\end{lemma}

\begin{proof}
Take $D = P_1 \cdots P_n$.
\end{proof}

For reducible $A$ and $B$ write $A = A_1 \cdots A_a$, where $A_i \in \mathbb{Z}[x]$ are irreducible, and similarly $B = B_1 \cdots B_b$. The polynomials $A_i$ and $B_j$ may be chosen to have integer coefficients by Gauss's lemma (see \cite{Lang}, Chapter IV, Section 2). For each $1 \le i \le a$ and $1 \le j \le b$ let $D_{i, j}$ be such that $S(D_{i, j}) = S(A_i) \cap S(B_j)$. By Lemma \ref{lem:union} there exists a polynomial $D$ such that
$$S(D) = \bigcup_{i, j} S(D_{i, j}).$$
We have $p \in S(D)$ if and only if $p \in S(A) \cap S(B)$, which proves Theorem \ref{thm:intersection}.

\begin{remark}
If one defines $m$ to be a divisor of $P$ if $m | P(n)$ for some integer $n$, then one can similarly prove, more strongly, that the common divisors of $A$ and $B$ are the divisors of $D$ for some $D$. By the Chinese remainder theorem this is reduced to handling prime powers, and for prime powers we have the following consequence of the Hensel's lifting lemma and Bezout's lemma for polynomials: if $P \in \mathbb{Z}[x]$ is irreducible, then for all but finitely many $p \in S(p)$ the equation $P(x) \equiv 0 \pmod{p^k}$ has a solution for all $k \ge 1$. The result immediately follows for reducible $P$ as well. One can prove in the spirit of Lemma \ref{lem:fix} that one may alter the highest power of $p$ which is a divisor of $P$ for any fixed number of $p$, except that one cannot set this highest power to be infinite. However, with enough care one can prove that in the weak form we get that any such ``strong'' prime divisor of $A$ and $B$ (i.e. a prime $p$ such that $p^k$ is a divisor of $A$ and $B$ for any $k$) is a strong prime divisor of $D$ too. (One may in the irreducible case handle only monic $A$ and $B$, and then apply a suitable transformation of the form $P(x) \to c^{\deg(P) - 1}P(\frac{x}{c})$ to reduce the general case to the monic case.)
\end{remark}

\section{Proof of Theorem \ref{thm:system}}

Ax (\cite{Ax}, Theorem 1) has proven the following statement. Let $F_1, F_2, \ldots , F_m$ be polynomials in the variables $x_1, \ldots , x_n$. The set of primes $p$ for which the system
\[
\begin{cases}
F_1(x_1, \ldots , x_n) \equiv 0 \pmod{p} \\
F_2(x_1, \ldots , x_n) \equiv 0 \pmod{p} \\
\vdots \\
F_m(x_1, \ldots , x_n) \equiv 0 \pmod{p} \\
\end{cases}
\]
has a solution may be expressed as a finite combination of unions, intersections and complements of sets of the form $S(P), P \in \mathbb{Z}[x]$.

In \cite{Dries} the result of Ax is refined by Dries as follows: one may drop the words ``unions'' and ``complements'' above, i.e. the set of $p$ for which the system is solvable is a finite intersection of sets of the form $S(P)$. By Theorem \ref{thm:intersection} this implies Theorem \ref{thm:system}.\footnote{It seems that the proof of Ax uses complemenets merely for modifying finite sets of primes in the sets of the form $S(P)$. By Lemma \ref{lem:fix} this can be done without using complements, and by Lemma \ref{lem:union} one can take unions of sets of the form $S(P)$. Thus, it would seem that this gives a different way to obtain the result of Dries than the one presented in \cite{Dries}.}

We think that the reader may benefit from an explanation \textit{why} the theorem should be true. We give a very high-level overview of the ideas.

Intuitively, if in Theorem \ref{thm:system} the number of equations $m$ is (much) larger than the number of variables $n$, we have too many constraints  and the system should have no solutions, at least in generic cases. For example, consider the case $m = 2$, $n = 1$. If $F_1, F_2 \in \mathbb{Z}[x_1]$ are coprime (which is the generic case), one may apply Bezout's lemma for polynomials to deduce that the system is solvable for only finitely many $p$. 

On the other hand, if $m$ is (much) smaller than $n$, we have quite a lot of freedom in the system, and intuitively there should be at least one solution to the system. Consider, for example, the case $m = 1, n = 2$. Turns out that a generic two-variable polynomial has roots modulo $p$ for all but finitely many primes $p$ (see the Lang-Weil bounds mentioned below). However, there are some examples of two-variable polynomials which do not have this property, such as $(xy)^2 + 1$.

The idea of a ``generic'' system of equation and its ``degrees of freedom'' is made precise via varieties. For each system $F$ of polynomial equations we define the so-called variety $V_F$ corresponding to the set of solutions of $F$. In this case, we are interested in varieties defined over the field of integers modulo $p$. If $V_F$ is irreducible (i.e. can't be expressed as a proper union of two varieties) and of dimension at least $1$ (one can think of this as $F$ having infinitely many complex solutions), the celebrated Lang-Weil bound gives that for all except finitely many primes $p$ the system $F$ has a solution modulo $p$. In general, the Lang-Weil bound gives a good approximation for the number of solutions of a system of polynomial congruneces. See \cite{Tao} for an accessible and more comprehensive exposition.

In the case $V_F$ has dimension $0$ (so there are only finitely many complex solutions to $F$), the idea is that these finitely many solutions of $F$ over $\mathbb{C}$ correspond to some finite vectors of algebraic numbers. Whether or not these algebraic numbers correspond to integers modulo $p$ is determined by whether or not $p$ is a prime divisor of a certain finite set of polynomials. See also Proposition 2.7 in \cite{Dries}.

Combining these observations, it sounds believable that Theorem \ref{thm:system} holds. Indeed, we have above proven that if the variety $V_F$ is irreducible (when considered over $\mathbb{C}$), the conclusion of Theorem \ref{thm:system} is true. Reducing the reducible case to the irreducible one is by no means trivial. We refer the reader to Ax's work \cite{Ax}.

\section{The Frobenius density theorem}

Before moving to the proofs of the remaining theorems we introduce the Frobenius density theorem.

Let $P \in \mathbb{Z}[x]$ be a non-constant polynomial whose roots are $\alpha_1, \ldots , \alpha_n$. The splitting field $F_P$ of $P$ is defined to be
$$\mathbb{Q}(\alpha_1, \ldots , \alpha_n).$$

We look at the isomorphisms $\sigma$ sending $F_P$ to itself. An isomorphism is an additive, multiplicative and bijective map sending $1$ to $1$. In the case $P(x) = x^2 - 2$ we have $F_P = \mathbb{Q}(\sqrt{2})$, and there are two such maps $\sigma$: the first is the identity sending $a + b\sqrt{2}$ to $a + b\sqrt{2}$ for all $a, b \in \mathbb{Q}$, and the second is the conjugation $a + b\sqrt{2} \to a - b\sqrt{2}$.

In the general case, a direct calculation gives that the roots of $P$ are mapped to each other. Indeed, if we write $P(x) = a_nx^n + \ldots + a_0$, we have for any $i$
$$P(\sigma(\alpha_i)) =  a_n\sigma(\alpha_i)^n + \ldots + a_1\sigma(\alpha_i) + a_0 = \sigma(a_n\alpha_i^n + \ldots + a_1\alpha_i + a_0) = \sigma(P(\alpha_i)) = \sigma(0) = 0,$$
so $\sigma(\alpha_i) = \alpha_j$ for some $j$. Since $\sigma$ is a bijection, it permutes the roots of $\alpha_i$. Furthermore, if we know $\sigma(\alpha_1), \ldots , \sigma(\alpha_n)$, then we know the values of $\sigma$ for all $x \in F_P$.

The set of these isomorphisms $\sigma$ is called the Galois group of $F_P$ (over $\mathbb{Q}$), which is denoted by $\Gal(F_P/\mathbb{Q})$. This is a group with respect to composition. For example, if $\sigma_1, \sigma_2 \in \Gal(F_P/\mathbb{Q})$, then the map $\sigma_3$ defined by $\sigma_3(x) = \sigma_1(\sigma_2(x))$ is also an isomorphism of $F_P$ to itself, and $\sigma_3$ is therefore in the Galois group.

Each $\sigma \in \Gal(F_P/\mathbb{Q})$ corresponds to a permutation, and we associate the permutation's cycle type to $\sigma$. The Frobenius density theorem connects these cycle types to the factorization of $P$ modulo $p$ for primes $p$. The statement is as follows. 

\textbf{Frobenius density theorem.} \textit{Let $P \in \mathbb{Z}[x]$ be non-constant and let $F_P$ be the splitting field of $P$. Let $a_1, \ldots , a_k$ be positive integers with sum $\deg(P)$. Let $N$ be the number of elements in $\Gal(F_P/\mathbb{Q})$ with cycle type $a_1, \ldots , a_k$, and let $S$ be the set of primes $p$ for which $P$ factorizes as the product of $k$ irreducible polynomials modulo $p$ with degrees $a_1, \ldots , a_k$. The density $\delta(S)$ of $S$ exists, and we have
$$\delta(S) = \frac{N}{|\Gal(F_P/\mathbb{Q})|}.$$}

For more discussion on the theorem and its proof, see \cite{SL}.

We are interested in applying the theorem in the case when some $a_i$ equals $1$, which corresponds to an element $\sigma$ fixing some root of $P$, as this gives the density of the prime divisors of $P$. 

As the proof of the Frobenius density theorem is not easy, we try to convince the reader of the fact that $S(P)$ has a positive density for non-constant $P$. Consider a prime $p$ such that $P$ is not constant modulo $p$. (This discards only finitely many $p$.) Now the equation $P(x) \equiv r \pmod{p}$ has at most $\deg(P)$ solutions for any $r$. Therefore, $P$ attains at least $p/\deg(P)$ values modulo $p$. Heuristically, the value $0$ should not be an especially unlikely value. Thus, one would predict that $\delta(S(P))$ is positive. In fact, the lower bound $\delta(S(P)) \ge \frac{1}{\deg(P)}$ suggested by this heuristic is true as well (\cite{BL}, Lemma 3), the result being a consequence of the Frobenius density theorem.

\section{Proof of Theorem \ref{thm:bound}}

As noted in the last section, the lower bound $\delta(S(P)) \ge \frac{1}{\deg(P)}$ for one polynomial follows from the Frobenius density theorem (\cite{BL}, Lemma 3). The lower bound for the general case follows by proving that the degree of the $D$ constructed in the proof of Theorem \ref{thm:intersection} equals $\deg(P_1) \cdots \deg(P_n)$. We consider the case $n = 2$ only, as the general case follows by induction.

Consider first the case when $A = P_1$ and $B = P_2$ are irreducible. Then, using the notation of the proof of Theorem \ref{thm:intersection}, we have
$$\deg(D) = \sum_{i = 1}^t \deg(D_i) = \sum_{i = 1}^t [\mathbb{Q}(\alpha, \beta_i) : \mathbb{Q}].$$
By using the fundamental multiplicative property
$$[\mathbb{Q}(\alpha, \beta_i) : \mathbb{Q}] = [\mathbb{Q}(\alpha, \beta_i) : \mathbb{Q}(\alpha)][\mathbb{Q}(\alpha) : \mathbb{Q}]$$
of the degrees of field extensions (\cite{Lang}, Chapter V, Proposition 1.2) we may simplify the above as
$$\deg(D) = \sum_{i = 1}^t \deg(A)\deg(E_i) = \deg(A)\deg(B).$$

Consider then the case of reducible $A$ and $B$. Let $A_1, \ldots , A_a$, $B_1, \ldots , B_b$ and $D_{i, j}$ be as in the proof. Now
\begin{align*}
\deg(D) = \deg\left(\prod_{i = 1}^{a} \prod_{j = 1}^{b} D_{i, j}\right) = \sum_{i = 1}^{a} \sum_{j = 1}^b \deg(A_i)\deg(B_j) \\
= \left(\sum_{i = 1}^a \deg(A_i)\right)\left(\sum_{j = 1}^b \deg(B_j)\right) = \deg(A)\deg(B).
\end{align*}

Here is an example of an equality case for the inequality of Theorem \ref{thm:bound}. Choose $P_i$ to be equal to the $k_i$th cyclotomic polynomial, where $k_1, \ldots , k_n$ are coprime integers. Now if $p \in S(P_i)$, then by a well-known property of the cyclotomic polynomials\footnote{The converse is true as well: if $p \equiv 1 \pmod{k_i}$, then $p$ is a prime divisor of $P_i$.  See for example \cite{MT}, Corollary 1 and Theorem 5.} we have either $p \equiv 1 \pmod{k_i}$ or $p \mid k_i$ (see for example \cite{MT}, Corollary 1), so by a quantitative version of Dirichlet's theorem
$$\delta\left(S(P_1) \cap \ldots \cap S(P_n)\right) \le \frac{1}{\phi(k_1 \cdots k_n)} = \frac{1}{\deg(P_1) \cdots \deg(P_n)}.$$
We remark that non-cyclotomic examples exist too, such as $n = 2, P_1(x) = x^3 - 2, P_2(x) = x^2 + x + 1$.

Note that this also proves that the construction of Theorem \ref{thm:intersection} is optimal (in the sense of the degree of $D$) in some cases: if $P_i$ are as above and $D$ is chosen such that $S(D) = S(P_1) \cap \ldots \cap S(P_n)$, then by the bound $\delta(S(D)) \ge \frac{1}{\deg(D)}$ we must have $\deg(D) \ge \deg(P_1) \cdots \deg(P_n)$.

\section{Proof of Theorem \ref{thm:product}}

The compositum field $F$ of Theorem \ref{thm:product} is defined to be the smallest extension of $\mathbb{Q}$ containing all of $F_i$. In this case $F$ is the splitting field of the product $P_1 \cdots P_n$. The condition of Theorem \ref{thm:product} corresponds to a certain kind of independence of the polynomials $P_i$.

Let $G = \Gal(F/\mathbb{Q})$ and $G_i = \Gal(F_i/\mathbb{Q})$. Consider the map $G \to G_1 \times G_2 \times \ldots \times G_n$ defined by
$$\sigma \to (\sigma | F_1, \ldots , \sigma | F_n).$$
Here $\sigma | F_i$ denotes the restriction of $\sigma \in G$ to the field $F_i$, so $\sigma | F_i$ is the map defined on $F_i$ which agrees with $\sigma$. This map is injective (\cite{Lang}, Chapter VI, Theorem 1.14).\footnote{In the reference the result is stated for $n = 2$ only, but the same proof works in the general case.} Since by the assumption\footnote{Note that the size of the Galois group equals the degree of the extension (\cite{Lang}, Chapter VI, Theorem 1.8)} we have $|G| = |G_1| \cdots |G_n|$, this map must be a bijection. This proves that $G$ is isomorphic to the product $G_1 \times G_2 \times \cdots \times G_n$. Now an application of the Frobenius density theorem yields the equality of the theorem, as the number of elements of $G$ fixing at least one root of each $P_i$ is the product of the corresponding quantities for the groups $G_i$, and the size of $G$ is equal to the product of the sizes of $G_i$.

\section{Proof of Theorem \ref{thm:surjective}}

Before the main proof we prove that there are many kinds of polynomials $P$ with $\delta(S(P)) = \frac{1}{n}$ for a fixed $n \ge 2$. This is done by the following three lemmas.

\begin{lemma}
\label{lem:subgroup}
Let $k \in \mathbb{Z_+}$. Let $S$ be the set of those $n$ for which $1 \le n \le k$ and for which $n$ and $k$ are coprime. Let $H$ be some subset of $S$ which is closed under multiplication, when multiplication is considered modulo $k$. There exists a polynomial $P \in \mathbb{Z}[x]$ with the following properties:

\begin{enumerate}
\item $p \in S(P)$ if and only if $p \equiv h \pmod{k}$ for some $h \in H$.
\item The splitting field $F_P$ of $P$ is a subfield of $\mathbb{Q}(\zeta_k)$, where $\zeta_k$ is a primitive $k$th root of unity.
\end{enumerate}

\end{lemma}

\begin{proof}
The result was first proven by Schur, but since the original work is difficult to find, we refer to the (modern) exposition given by Ram Murty and Nithum Thain in \cite{MT}. The first property of the lemma follows from their Theorems 4 and 5, and the second property follows from the construction. Note that in Theorem 4 of \cite{MT} it is stated that there may be a finite number of expectations, but these can be handled by Lemma \ref{lem:fix}.
\end{proof}

\begin{lemma}
\label{lem:1/n}
Let $n \ge 2$ be an integer. There exists infinitely many primes $p$ for which there exists a $P \in \mathbb{Z}[x]$ satisfying the following conditions:

\begin{enumerate}
\item $\delta(S(P)) = \frac{1}{n}$.
\item The splitting field $F_P$ of $P$ is a subfield of $\mathbb{Q}(\zeta_p)$.
\end{enumerate}

\end{lemma}

\begin{proof}
By Dirichlet's theorem there exists infinitely many primes $p$ such that $p \equiv 1 \pmod{n}$. Pick such a $p$ and write $p = mn + 1$, $m \in \mathbb{Z}$. Let $H$ be the unique subgroup of $(\mathbb{Z}/p\mathbb{Z})^*$ (the non-zero elements modulo $p$) of size $m$. Let $P$ be as in Lemma \ref{lem:subgroup}, so $p \in S(P)$ if and only if $p \equiv h \pmod{p}$ for some $h \in H$. By Dirichlet's theorem we have $\delta(S(P)) = \frac{|H|}{p-1} = \frac{m}{p-1} = \frac{1}{n}$. Thus, property 1 is satisfied for this $P$. Property 2 follows from the choice of $P$ and Lemma \ref{lem:subgroup}.
\end{proof}

The following lemma provides a way to ``merge'' to a polynomial $P$ a polynomial $Q$ with $\delta(S(Q)) = \frac{1}{n}$ such that $P$ and $Q$ are independent in the sense of Theorem \ref{thm:product}.

\begin{lemma}
\label{lem:combine}
Let $P \in \mathbb{Z}[x]$ and $n \ge 2$ be arbitrary. There exists a $Q \in \mathbb{Z}[x]$ which satisfies the following conditions:

\begin{enumerate}
\item $\delta(S(Q)) = \frac{1}{n}.$
\item $F_P \cap F_Q = \mathbb{Q}$, where $F_R$ denotes the splitting field of $R$.
\end{enumerate}

\end{lemma}

\begin{proof}
By Lemma \ref{lem:1/n} there exists arbitrarily large primes $p$ and polynomials $Q_p$ for which $\delta(S(Q_p)) = \frac{1}{n}$ and $F_{Q_p} \subset \mathbb{Q}(\zeta_p)$. It suffices to prove that at least one of these $Q_p$ is such that its splitting field does not intersect the splitting field of $P$ non-trivially. Assume the contrary. Thus, for infinitely many primes $p$ we have $F_P \cap \mathbb{Q}(\zeta_p) \neq \mathbb{Q}$.  

The fundamental theorem of Galois theory tells that there is a correspondence between the subgroups of $\Gal(F_P/\mathbb{Q})$ and subfields of $F_P$. The theorem tells that, in particular, the number of subfields of $F_P$ is finite (\cite{Lang}, Chapter VI, Corollary 1.6). Therefore, there exists two distinct primes $p_1, p_2$ for which $F_P \cap \mathbb{Q}(\zeta_{p_1}) = F_P \cap \mathbb{Q}(\zeta_{p_2}) = K \neq \mathbb{Q}$. This is a contradiction, as we have $\mathbb{Q}(\zeta_m) \cap \mathbb{Q}(\zeta_n) = \mathbb{Q}$ for any coprime integers $m, n$ (\cite{Lang}, Chapter VI, Corollary 3.2).
\end{proof}

We now prove by induction with respect to $k$ that for all $\frac{m}{k}$ there exists a polynomial $P$ with $\delta(S(P)) = \frac{m}{k}$, when $0 \le m \le k$, $m, k \in \mathbb{Z}$, $k \ge 1$. The case $k = 1$ is clear, so assume the statement holds for $k = n-1$, and prove it for $k = n \ge 2$. Thus, we want to prove that for all $0 \le m \le n$ there exists a polynomial for which its density of prime divisors is $\frac{m}{n}$. This is clear for $m = 0$, so assume $m > 0$. By the induction hypothesis there exists a polynomial $P$ for which $\delta(S(P)) = \frac{m-1}{n-1}$.

Let $Q$ be as in Lemma \ref{lem:combine} so that $\delta(S(Q)) = \frac{1}{n}$ and $F_P \cap F_Q = \mathbb{Q}$. The latter property is equivalent to the condition of Theorem \ref{thm:product} (see \cite{Lang}, Chapter VI, Theorem 1.14 and the proof of Theorem \ref{thm:product}), so we have $\delta(S(P) \cap S(Q)) = \delta(S(P))\delta(S(Q)) = \frac{m-1}{n(n-1)}$. A direct calculation now gives
$$\delta(S(PQ)) = \delta(S(P) \cup S(Q)) = \delta(S(P)) + \delta(S(Q)) - \delta(S(P) \cap S(Q)) = \frac{m}{n}.$$
This completes the induction, thus proving Theorem \ref{thm:surjective}.

\end{document}